\documentclass[11pt,a4paper]{amsart}

 \usepackage[utf8x]{inputenc}
\usepackage{latexsym}
\usepackage{color,graphicx,shortvrb}
\usepackage{amsmath, amssymb}
\usepackage{amsfonts}
\usepackage[colorlinks, bookmarks=true]{hyperref}

\newtheorem{theorem}{Theorem}[section]

\newtheorem{proposition}[theorem]{Proposition}
\newtheorem{corollary}[theorem]{Corollary}

\theoremstyle{definition}

\author{J. M. Almira}
\title{On Popoviciu-Ionescu functional equation}

\begin{document}
\keywords{Functional equations, Exponential polynomials on Abelian groups, Montel type theorem}


\subjclass[2010]{39B22, 39A70, 39B52}

\address{Departamento de Matem\'{a}ticas, Universidad de Ja\'{e}n, E.P.S. Linares,  Campus Cient\'{\i}fico Tecnol\'{o}gico de Linares, Cintur\'{o}n Sur s/n, 23700 Linares, Spain}
\email{jmalmira@ujaen.es}


\begin{abstract}
We study a functional equation first proposed by  T. Popoviciu \cite{P} in 1955. It was solved for the easiest case by Ionescu \cite{I} in 1956  and, for the general case, by  Ghiorcoiasiu and Roscau \cite{GR}  and Rad\'{o} \cite{R} in 1962. Our solution is based on a generalization of Rad\'{o}'s theorem to distributions in a higher dimensional setting  and, as far as we know, is different than existing solutions. Finally, we propose several related open problems.
\end{abstract}

\maketitle

\markboth{J. M. Almira}{Popoviciu-Ionescu functional equation}

\section{Motivation}
We study the continuous solutions of the functional equation
\begin{equation}\label{Po}
\det\left[\begin{array}{cccccc}
f(x) & f(x+h) &  \cdots & f(x+nh) \\
f(x+h) & f(x+2h) & \cdots & f(x+(n+1)h)\\
\vdots & \vdots & \ \ddots & \vdots \\
f(x+nh) & f(x+(n+1)h) &  \cdots &  f(x+2nh)\\
\end{array} \right] =0 \text{ for all } x,h\in\mathbb{R}.
\end{equation}
This equation was proposed by T. Popoviciu \cite{P} for functions $f:\mathbb{R}\to\mathbb{R}$ and was studied by several Romanian mathematicians in the 1960's \cite{C,I,GR,R,S}. In particular for the case of  continuous functions $f:\mathbb{R}\to\mathbb{R}$, Iounescu \cite{I} solved it for $n=1,2$ and, later on, as a result of the joint efforts of Ghiorcoiasiu and Roscau \cite{GR} and Rad\'{o} \cite{R}, it was solved for arbitrary $n$. Concretely, Ghiorcoiasiu and Roscau proved that, if $f:\mathbb{R}\to\mathbb{R}$ is a continuous solution of \eqref{Po}, then there exist $H>0$ and continuous  functions $a_k:(0,H)\to \mathbb{R}$, $k=0,\cdots,n$, such that $(a_0(h),\cdots,a_n(h))\neq (0,\cdots,0)$ for some $h\in (0,H)$ and
\begin{equation}\label{rado}
a_0(h)f(x)+a_1(h)f(x+h)+\cdots +a_n(h)f(x+nh)=0 \text{ for all } x\in\mathbb{R} \text{ and all } 0\leq h<H. 
\end{equation}
and Rad\'{o} proved that, for continuous functions $f:\mathbb{R}\to\mathbb{R}$, the equation \eqref{rado} characterizes the exponential polynomials which solve an ordinary homogenous linear differential equation of order $n$, with constant coefficients, $A_0f+A_1f'+\cdots+A_nf^{(n)}=0$. These equations are, furthermore, strongly connected to Levi-Civita's functional equation
\begin{equation}\label{LCn}
F(x+y)=\sum_{k=1}^n\varphi_k(x)\phi_k(y),
\end{equation}
which may be studied on a much more general setting of functions defined on groups or semigroups (see, for example, the monographs by Stetkaer  \cite{St} and Sz{\'e}kelyhidi \cite{Sz}, or the papers by Shulman \cite{E1,E2,E3,E4,E5,E6}). 

For functional equations like \eqref{rado}, which can be viewed as depending on a parameter $h$, it makes sense to ask about the minimal sets of parameters  
$\{h_i\}_{i\in I}$ with the property that, if $f$ solves the equation with $h=h_i$ for all $i\in I$, then it solves the equation for all $h$. These kind of results are named Montel-type theorems after the seminal papers by the French mathematician Montel, who studied the problem for Fr\'{e}chet's functional equation $\Delta_h^{n}f=0$ \cite{M1,M2,M3} (see also \cite{A1,A2,A3,A4}). 
 
In this note we use Anselone-Korevaar's theorem \cite{anselone} for a study of a Montel-type theorem connected to Rad\'{o}'s functional equation \eqref{rado}, which we re-formulate for distributions defined on $\mathbb{R}^d$, and use the corresponding result to give a new proof of the fact that continuous solutions $f:\mathbb{R}\to\mathbb{R}$ of \eqref{Po} are exponential polynomials.  Finally, we propose several open problems for the higher dimensional setting.


\section{Main result}

\begin{theorem} \label{main} Let $h_1, h_2,\dots,h_s$  be such that they span a dense additive subgroup
of $\mathbb{R}^d$. Let $f$ be a distribution on $\mathbb{R}^d$ such that there exist natural
numbers $n_i$, $i=1,\dots,s$  satisfying
\begin{equation}\label{uno}
  \dim \mathbf{span} \{f, \tau_{h_i}(f),\cdots, (\tau_{h_i})^{n_i}(f)\} \leq n_i, \  \  i=1,\cdots,s.
\end{equation}
Then $f$ is, in distributional sense, a continuous exponential polynomial. In particular, $f$ is an ordinary function which is equal almost everywhere, in the Lebesgue measure, to an exponential polynomial.
\end{theorem}

\begin{proof} If $f=0$ we are done. Thus, we impose $f\neq 0$.  Let us assume, without loss of generality,  that $n_i$ is the smallest natural number satisfying
\eqref{uno}, and let 
$$W_i=\mathbf{span} \{f, \tau_{h_i}(f),\cdots, (\tau_{h_i})^{n_i-1}(f)\}.$$
Obviously, $\dim(W_i)=n_i$ (otherwise $n_i$ would not be minimal). Furthermore,
\eqref{uno}  implies that $\tau_{h_i}(W_i)\subseteq W_i$. Hence $\tau_{h_i}$ defines an automorphism on $W_i$, since $\tau_{h_i}$ is always injective and $\dim W_i<\infty$.   Consequently, $(\tau_{h_i})^m(W_i)= W_i $ for all integral numbers $m$. In  particular, for each $m\in\mathbb{Z}$ there exist numbers $\{a_{i,m,k}\}_{k=0}^{n_i-1}$ such that
\[
\tau_{h_i}^m(f)=\sum_{k=0}^{n_i-1}a_{i,m,k}(\tau_{h_i})^k(f).
\]
It follows that 
\begin{eqnarray*}
&\ & \tau_{m_1h_1+\cdots+m_sh_s}(f) = (\tau_{h_1})^{m_1}(\tau_{h_2})^{m_2}\cdots (\tau_{h_s})^{m_s}(f)\\
&\ & \  \  = (\tau_{h_1})^{m_1}\cdots(\tau_{h_{s-1}})^{m_{s-1}}\left (\sum_{k_s=0}^{n_s-1}a_{s,m_s,k_s}(\tau_{h_s})^{k_s}(f)\right)\\
&\ & \  \  = \sum_{k_s=0}^{n_s-1} a_{s,m_s,k_s} (\tau_{h_1})^{m_1}\cdots(\tau_{h_{s-1}})^{m_{s-1}} (\tau_{h_s})^{k_s}(f)\\
&\ & \  \  = \sum_{k_s=0}^{n_s-1} a_{s,m_s,k_s} (\tau_{h_s})^{k_s} (\tau_{h_1})^{m_1}\cdots(\tau_{h_{s-1}})^{m_{s-1}} (f)\\
&\ & \  \  = \sum_{k_s=0}^{n_s-1} a_{s,m_s,k_s} (\tau_{h_s})^{k_s} (\tau_{h_1})^{m_1}\cdots(\tau_{h_{s-2}})^{m_{s-2}} \left (\sum_{k_{s-1}=0}^{n_{s-1}-1}a_{s-1,m_{s-1},k_{s-1}}(\tau_{h_{s-1}})^{k_{s-1}}(f)\right)\\
&\ & \  \  = \sum_{k_s=0}^{n_s-1} \sum_{k_{s-1}=0}^{n_{s-1}-1} a_{s,m_s,k_s} a_{s-1,m_{s-1},k_{s-1}}(\tau_{h_s})^{k_s} (\tau_{h_{s-1}})^{k_{s-1}} (\tau_{h_1})^{m_1}\cdots(\tau_{h_{s-2}})^{m_{s-2}}(f),
\end{eqnarray*}
and, repeating the argument $s$ times, we get
\[
\tau_{m_1h_1+\cdots+m_sh_s}(f) = \sum_{k_s=0}^{n_s-1} \sum_{k_{s-1}=0}^{n_{s-1}-1} \cdots \sum_{k_{1}=0}^{n_{1}-1} a_{s,m_s,k_s} a_{s-1,m_{s-1},k_{s-1}}\cdots a_{1,m_{1},k_{1}}(\tau_{h_s})^{k_s}  \cdots (\tau_{h_1})^{k_1}(f).
\]
In other words, if we consider the space
$$W=\mathbf{span} \{f, (\tau_{h_1})^{a_1}(\tau_{h_2})^{a_2}\cdots (\tau_{h_s})^{a_s}(f): 0\leq a_i<n_i, i=1,2,\cdots,s\}, $$
then
$$\tau_{m_1h_1+\cdots+m_sh_s}(f)\in W \text{ for all } (m_1,\cdots,m_s) \in \mathbb{Z}^s.$$
Hence every translation of $f$ belongs to $W$, since $h_1,\cdots,h_s$ span a dense additive subgroup of $\mathbb{R}^d$ and $W$ is finite
dimensional. The proof ends by applying Anselone-Korevaar's theorem.  
\end{proof}

\begin{proposition}
Every open subset $V$ of  $\mathbb{R}^d$ contains a finite set of vectors $\{h_0,\cdots,h_s\}$ which span a dense subgroup of $\mathbb{R}^d$.
\end{proposition}

\begin{proof} A well known result by Kronecker states that $\mathbb{Z}^d+(\theta_1,\theta_2,\cdots,\theta_d)\mathbb{Z}$ is dense in $\mathbb{R}^d$ if and only if 
$\{1,\theta_1,\cdots,\theta_d\}$ is $\mathbb{Q}$-linearly independent (see \cite[Theorem 442, page 382]{HW}).  Of course, the same claim holds true for the subgroup $\frac{1}{N}\mathbb{Z}^d+(\theta_1,\theta_2,\cdots,\theta_d)\mathbb{Z}$ for every $N>0$. Hence, every open neighborhood of $(0,0,\cdots,0)$ contains a finite set of vectors $\{h_0,\cdots,h_s\}$ which span a dense subgroup of $\mathbb{R}^d$.

Let  $V$ be any open subset of  $\mathbb{R}^d$ and let  $x_0\in V$ and $\varepsilon >0$ be such that $x_0+B_0(\varepsilon)\subseteq V$.  Take $\{h_1,\cdots,h_s\}\subseteq B_0(\varepsilon) $  such that $h_1\mathbb{Z}+\cdots+h_s\mathbb{Z}$ is  dense in $\mathbb{R}^d$. Then $\{x_0, x_0+h_1,\cdots, x_0+h_s\}\subset V$  spans a dense additive subgroup   $\mathbb{R}^d$. 
\end{proof}

\begin{corollary}[Rad\'{o}'s theorem for higher dimensions] \label{rado_d} Assume that $f:\mathbb{R}^d\to\mathbb{R}$ is a continuous solution of 
\begin{equation}\label{radod}
a_0(h)f(x)+a_1(h)f(x+h)+\cdots +a_n(h)f(x+nh)=0 \text{ for all } x\in\mathbb{R}^d \text{ and all } h\in U 
\end{equation}
for a certain open set $U\subseteq \mathbb{R}^d$ and certain continuous functions $a_k:U\to\mathbb{R}$ such that $a=(a_0,\cdots,a_n)$ does not   vanish identically. Then $f$ is an exponential polynomial in $d$ variables. 
\end{corollary}

\begin{proof}
Let $h_0\in U$ be such that $a(h_0)\neq (0,\cdots,0)$. The continuity of $a$ implies that $a(h)\neq (0,\cdots,0)$ for all $h\in V$ for a certain open set $V\subseteq U$. Let us take $\{h_1,\cdots, h_s\}\subset V$ spanning a dense subgroup $h_1\mathbb{Z}+\cdots+h_s\mathbb{Z}$ of $\mathbb{R}^d$. Then   \eqref{radod} implies that 
\begin{equation}\label{dos}
  \dim \mathbf{span} \{f, \tau_{h_i}(f),\cdots, (\tau_{h_i})^{n}(f)\} \leq n, \  \  i=1,\cdots,s,
\end{equation}
and Theorem \ref{main} implies that $f$ is equal almost everywhere to an exponential polynomial.  Hence $f$ itself is an exponential polynomial, since $f$ is continuous.
\end{proof}

\begin{corollary}
Let $f:\mathbb{R}\to\mathbb{R}$ be a continuous function which solves \eqref{Po} for all $x,h\in\mathbb{R}$. Then $f$ is an exponential polynomial.
\end{corollary}
\begin{proof}  Ghiorcoiasiu and Roscau's theorem  \cite{GR} guarantees that $f$ solves \eqref{rado} for some continuous function $a(h)=(a_0(h),\cdots,a_n(h))$ such that   $a(h_0)\neq (0,\cdots,0)$ for some $h_0\in (0,H)$. Indeed, they prove that we can impose $a_n(h)=1$ for all $h\in (0,H)$ (see \cite[Theorem 5]{GR}). The result follows just applying Corollary \ref{rado_d} to $f$.
\end{proof}

It is natural to ask what are the continuous solutions of Popoviciu-Ionescu functional equation \eqref{Po} in the higher dimensional setting. This problem seems to be still open. In particular, the technique used by Ghiorcoiasiu and Roscau to reduce this equation to equation \eqref{rado} seems to fail in this context, since the proof of Theorem 3 of their paper \cite{GR} strongly depends on the fact that all arguments live in the very same line. Hence, without a new proof of a result of that kind, we can't use Corollary \ref{rado_d} in this context.  Does this mean that Popoviciu-Ionescu's functional equation admits non-exponential polynomial continuous solutions in the higher dimensional context? We do not believe it, but a proof is still far away from being at our hands. Consequently, we  state the following 
\vspace{0.5cm}

\noindent \textbf{Open Problem 1.}  Is it true  that all continuous solutions of the equation \begin{equation}\label{Po_d}
\det\left[\begin{array}{cccccc}
f(x) & f(x+h) &  \cdots & f(x+nh) \\
f(x+h) & f(x+2h) & \cdots & f(x+(n+1)h)\\
\vdots & \vdots & \ \ddots & \vdots \\
f(x+nh) & f(x+(n+1)h) &  \cdots &  f(x+2nh)\\
\end{array} \right] =0 \text{ for all } x,h\in\mathbb{R}^d.
\end{equation}
are exponential polynomials? 
\vspace{0.5cm}

By the way, we know that every exponential polynomial $f:\mathbb{R}^d\to\mathbb{C}$ solves  the equation \eqref{Po_d} for all $n$ large enough. Indeed, if $f$ is an exponential polynomial then $\tau(f)=\mathbf{span}\{\tau_hf:h\in\mathbb{R}^d\}$ is a finite dimensional space. Hence, if $n=\dim \tau(f)$ and $h\in\mathbb{R}^d$, there exist coefficients $a_k(h)$ such that   $a_0(h)f+a_1(h)\tau_h(f)+\cdots+a_n(h)\tau_{nh}(f)$ vanishes identically and, henceforth, $f$ solves  \eqref{Po_d}.

On the other hand, we can demonstrate the following (almost trivial) result:

\begin{proposition} \label{pro} Assume that $f:\mathbb{R}^d\to\mathbb{C}$ is a continuous solution of \eqref{Po_d}. Then $f$, restricted to any line $L$, defines an exponential polynomial.
\end{proposition}
\begin{proof} Given $x_0,h_0\in \mathbb{R}^d$, we set $F_{x_0,h_0}(t)=f(x_0+th_0)$. Then $F_{x_0,h_0}$ is a continuous solution of  \eqref{Po}, so that it is an exponential polynomial.  
\end{proof}

The result above motivates the statement of another question: 

\vspace{0.5cm}

\noindent \textbf{Open Problem 2.}  Assume that $f:\mathbb{R}^d\to\mathbb{R}$, restricted to any line $L$, defines an exponential polynomial,  which means that all functions $F_{x_0,h_0}(t)=f(x_0+th_0)$ satisfy the Levi-Civita functional equation \eqref{LCn} for some $n$ and some functions $\varphi_k,\phi_k$, $k=1,\cdots,n$.  Is it true, then, that $f$ is itself an exponential polynomial? The problem can be stated either for arbitrary functions $f$, in which case being an exponential polynomial should be understood as being a solution of Levi-Civita functional equation in $\mathbb{R}^d$, or for functions $f$ satisfying some restriction, like being continuous, in which case the exponential polynomials are just finite linear combinations of exponential monomials $x^{\alpha}e^{\langle x,\lambda \rangle}$, with  $\alpha\in\mathbb{N}^d$ and $\lambda\in\mathbb{C}^d$.  

\vspace{0.5cm}

For polynomial functions, a  result of this type was demonstrated by Prager and Schwaiger in 2009 \cite[Theorem 14]{PS}. Concretely, they proved that  if  $K$ is a field and $ f : K^d\to K $ is an ordinary algebraic polynomial function separately in each variable (which means that for any $1\leq k\leq d$ and any point $(a_1,\cdots,a_{k-1},a_{k+1},\cdots,a_d)\in K^{d-1}$, the function $f(a_1,\cdots,a_{k-1},x_k,a_{k+1},\cdots,a_d)$ is an ordinary algebraic polynomial in $x_k$) then $f$ is an ordinary algebraic polynomial function in $d$ variables provided that $K$ is finite or uncountable. Furthermore, for every countable infinite field $K$ there exists a function $f:K^2\to K$ which is an ordinary algebraic polynomial function separately in each variable and is not a generalized polynomial in both variables jointly. Of course, the result does not assume continuity of $f$ nor any common upper bound for the degrees of the polynomials $f(a_1,\cdots,a_{k-1},x_k,a_{k+1},\cdots,a_d)$. It turns out that a similar  result can be demonstrated for trigonometric polynomials:
\begin{theorem} \label{PST} Let  $f:\mathbb{R}^d\to\mathbb{C}$  be a function satisfying  that there exist  $T_1,\cdots,T_d>0$ such that for any $1\leq k\leq d$ and any point 
$(a_1,\cdots,a_{k-1},a_{k+1},\cdots,a_d)\in \mathbb{R}^{d-1}$, the function $f(a_1,\cdots,a_{k-1},x_k,a_{k+1},\cdots,a_d)$ is a $T_k$-periodic trigonometric polynomial in $x_k$. Then $f=P(e^{\frac{2\pi ix_1}{T_1}},e^{-\frac{2\pi ix_1}{T_1}},\cdots,e^{\frac{2\pi ix_d}{T_d}},e^{-\frac{2\pi ix_d}{T_d}})$ for a certain ordinary algebraic polynomial $P$ (i.e., $f$ is a trigonometric polynomial of several variables).
\end{theorem}
\begin{proof} The proof follows the very same steps of the demonstration of \cite[Theorem 14]{PS}. We include it here for the sake of completeness. We proceed by induction on the dimension $d$. For $d=1$ there is nothing to prove. Let us assume the result holds for $d-1$ variables and let us now assume that our function depends on $d$ variables. For each $b_d\in [0,T_d)$, the induction hypothesis confirms us that 
\[
f(b_1,\cdots,b_{d-1},b_d)=\sum_{\alpha\in\mathbb{Z}^{d-1}}A_{\alpha}(b_d) e^{\alpha_1\frac{2\pi ib_1}{T_1}}\cdots e^{\alpha_{d-1}\frac{2\pi ib_{d-1}}{T_d}}
\]
for certain functions $A_{\alpha}$ with the property that, for each $\xi \in [0,T_d)$, $I(\xi)=\#\{\alpha: A_{\alpha}(\xi)\neq 0\}<\infty$. Given $p\in\mathbb{N}$ we define $F_p=\{\xi\in[0,T_d):I(\xi)\subseteq \{-p,-p+1,\cdots,-1,0,1,2,\cdots, p\}^{d-1}\}$. Obviously, $[0,T_d)=\bigcup_pF_p$ and $[0,T_d)$ is uncountable. Hence $F_m$ is uncountable for some  $m$.  In particular, $F_m$ is infinite and 
\[
f(b_1,\cdots,b_{d-1},b_d)=\sum_{\alpha\in\{-m,\cdots,-1,0,1,\cdots,m\}^{d-1}}A_{\alpha}(b_d) e^{\alpha_1\frac{2\pi ib_1}{T_1}}\cdots e^{\alpha_{d-1}\frac{2\pi ib_{d-1}}{T_d}}
\]
for all $(b_1,\cdots,b_{d-1})\in [0,T_1)\times [0,T_2)\times \cdots \times[0,T_{d-1})$ and for all $b_d\in F_m$. 
We can choose some sets of points $Q_k\subseteq [0,T_k)$, $k=1,\cdots,d-1$, with cardinality $\#Q_k=2m+1$ such that, if we set $Q=Q_1\times Q_2\times\cdots Q_{d-1}$,  the system of linear equations (in the unknowns $u_{\alpha}$)
\[
f(y_1,\cdots,y_{d-1},b_d)=\sum_{\alpha\in\{-m,\cdots,-1,0,1,\cdots,m\}^{d-1}}u_{\alpha}e^{\alpha_1\frac{2\pi iy_1}{T_1}}\cdots e^{\alpha_{d-1}\frac{2\pi iy_{d-1}}{T_d}} ,\ \ (y_1,\cdots,y_{d-1})\in Q
\]
admits, for each $b_d\in F_m$, a unique solution $$u_{\alpha}(b_d)=\sum_{(y_1,\cdots,y_{d-1})\in Q} c_{\alpha}(y_1,\cdots,y_{d-1})f(y_1,\cdots,y_{d-1},b_d).$$ 
Hence 
\[
a_{\alpha}(x_d)= \sum_{(y_1,\cdots,y_{d-1})\in Q} c_{\alpha}(y_1,\cdots,y_{d-1})f(y_1,\cdots,y_{d-1},x_d)
\]
defines an ordinary algebraic polynomial in $\{e^{\frac{2\pi i x_d}{T_d}},e^{-\frac{2\pi i x_d}{T_d}}\}$ (since this is the case for all functions $f(y_1,\cdots,y_{d-1},x_d)$) and satisfies the identities  $a_{\alpha}(b_d)=A_{\alpha}(b_d)$ for all $b_d\in F_m$ and all $\alpha\in \{-m,\cdots,-1,0,1,\cdots,m\}^{d-1}$.  Let us now consider the trigonometric polynomial 
\[
g(x_1,\cdots,x_{d-1},x_d)=\sum_{\alpha\in\{-m,\cdots,-1,0,1,\cdots,m\}^{d-1}}a_{\alpha}(x_d) e^{\alpha_1\frac{2\pi ix_1}{T_1}}\cdots e^{\alpha_{d-1}\frac{2\pi ix_{d-1}}{T_d}}
\]
Obviously, $g$ can be written as 
\[
g(x_1,\cdots,x_{d-1},x_d)=\sum_{k=-r}^rg_k(e^{\frac{2\pi i x_1}{T_1}},e^{-\frac{2\pi i x_1}{T_1}}, \cdots, e^{\frac{2\pi i x_{d-1}}{T_{d-1}}},e^{-\frac{2\pi i x_{d-1}}{T_{d-1}}}) e^{k\frac{2\pi i x_d}{T_d}} 
\]
for certain ordinary algebraic polynomials $g_k(X_1,\cdots,X_{2d-2})$ and certain $r>0$. On the other hand, there exist functions $f_k$ such that 
\[
f(b_1,\cdots,b_{d-1},b_d)=\sum_{k=-\infty}^{\infty}f_k(e^{\frac{2\pi i x_1}{T_1}},e^{-\frac{2\pi i x_1}{T_1}}, \cdots, e^{\frac{2\pi i x_{d-1}}{T_{d-1}}},e^{-\frac{2\pi i x_{d-1}}{T_{d-1}}}) e^{k\frac{2\pi i b_d}{T_d}}, 
\]
where, for each $(b_1,\cdots,b_{d-1})\in [0,T_1)\times \cdots\times [0,T_{d-1})$, the number of $k$'s such that $$f_k(e^{\frac{2\pi i x_1}{T_1}},e^{-\frac{2\pi i x_1}{T_1}}, \cdots, e^{\frac{2\pi i x_{d-1}}{T_{d-1}}},e^{-\frac{2\pi i x_{d-1}}{T_{d-1}}}) \neq 0$$ is finite. 
Now, given $(b_1,\cdots,b_{d-1})$, the equality $f(b_1,\cdots,b_{d-1},b_d)=g(b_1,\cdots,b_{d-1},b_d)$ holds true for infinitely many points $b_d\in F_m$. This proves that $$f_k(e^{\frac{2\pi i x_1}{T_1}},e^{-\frac{2\pi i x_1}{T_1}}, \cdots, e^{\frac{2\pi i x_{d-1}}{T_{d-1}}},e^{-\frac{2\pi i x_{d-1}}{T_{d-1}}}) =g_k(e^{\frac{2\pi i x_1}{T_1}},e^{-\frac{2\pi i x_1}{T_1}}, \cdots, e^{\frac{2\pi i x_{d-1}}{T_{d-1}}},e^{-\frac{2\pi i x_{d-1}}{T_{d-1}}}) $$ for  $|k|\leq r$ and $$f_k(e^{\frac{2\pi i x_1}{T_1}},e^{-\frac{2\pi i x_1}{T_1}}, \cdots, e^{\frac{2\pi i x_{d-1}}{T_{d-1}}},e^{-\frac{2\pi i x_{d-1}}{T_{d-1}}}) =0 $$ for $|k|>r$. Henceforth, $f=g$, which ends the proof.  
\end{proof}

Now we can state the following result, which is just a first step for the study of continuous solutions of  \eqref{Po_d}. 

\begin{theorem} Assume that $f:\mathbb{R}^d\to\mathbb{C}$ is a continuous solution of \eqref{Po_d}.  If  there exist  $T_1,\cdots,T_d>0$ such that for any $1\leq k\leq d$ and any point  $(a_1,\cdots,a_{k-1},a_{k+1},\cdots,a_d)\in \mathbb{R}^{d-1}$, the function $f(a_1,\cdots,a_{k-1},x_k,a_{k+1},\cdots,a_d)$ is $T_k$-periodic, then  $$f=P(e^{\frac{2\pi ix_1}{T_1}},e^{-\frac{2\pi ix_1}{T_1}},\cdots,e^{\frac{2\pi ix_d}{T_d}},e^{-\frac{2\pi ix_d}{T_d}})$$ for a certain ordinary algebraic polynomial $P$.
\end{theorem}

\begin{proof} Given $k\in \{1,\cdots,d\}$ and  $(a_1,\cdots,a_{k-1},a_{k+1},\cdots,a_d)\in\mathbb{R}^{d-1}$,  Proposition \ref{pro}, and our assumptions on the periodicity of $f(a_1,\cdots,a_{k-1},x_k,a_{k+1},\cdots,a_d)$, guarantee that $$f(a_1,\cdots,a_{k-1},x_k,a_{k+1},\cdots,a_d)=Q(e^{\frac{2\pi ix_k}{T_k}},e^{-\frac{2\pi ix_k}{T_k}} )$$ for a certain ordinary algebraic polynomial $Q$. Now we can apply Theorem \ref{PST}.
\end{proof}

As a particular case of Open Problem 2, we state the following
\vspace{0.5cm}

\noindent \textbf{Open Problem 3. } Is it true that all bounded continuous solutions of \eqref{Po_d} are bounded exponential polynomials in $\mathbb{R}^d$, which is the same as saying that they are finite sums  of the form $$\sum_{k=1}^mP_k(e^{\frac{2\pi ix_1}{T_{k,1}}},e^{-\frac{2\pi ix_1}{T_{k,1}}},\cdots,e^{\frac{2\pi ix_d}{T_{k,d}}},e^{-\frac{2\pi ix_d}{T_{k,d}}}),$$
with each $P_i$ being an ordinary algebraic polynomial in $2d$ variables? In other words, we wonder if they are finite linear combinations of functions of the form $e^{\langle \lambda i, \alpha x\rangle}$, with $\lambda\in\mathbb{R}^d$, $\alpha\in\mathbb{Z}^d$, $x=(x_1,\cdots,x_d)$. Of course, as we have already observed, all these functions are  solutions of  \eqref{Po_d} for some $n\in\mathbb{N}$, since they solve a Levi-Civita functional equation.  

\section{Acknowledgement} The author is deeply thankful to the anonymous referees, since they helped to improve the readability of this paper. One of them suggested the correct formulation (and proof) of Theorems 2.6 and 2.7, and Open Problem 3.   
\bibliographystyle{amsplain}


\bigskip



\end{document}